\documentclass[11pt,a4paper, fleqn]{article}

\RequirePackage[OT1]{fontenc}
\RequirePackage{amsthm,amsmath, amssymb, enumerate}
\RequirePackage[numbers]{natbib}
\RequirePackage[colorlinks,citecolor=blue,urlcolor=blue]{hyperref}
\RequirePackage{a4wide, hypernat}
\usepackage{authblk}
\usepackage{setspace}

\numberwithin{equation}{section}
\theoremstyle{plain}
\newtheorem{theorem}{Theorem}[section]

\newtheorem{definition}{Definition}[section]

\newtheorem{corollary}{Corollary}[section]
\newtheorem{proposition}{Proposition}[section]
\newtheorem{remark}{Remark}[section]

\newcommand{\R}{\mathbb R}
\newcommand{\N}{\mathbb N}

\newcommand{\Z}{\mathbb{Z}}

\begin{document}

\title{Parametric Stein operators and variance bounds}

\author{Christophe Ley\footnote{D\'epartement de Math\'ematique, Campus Plaine, Boulevard du Triomphe, CP210, B-1050 Brussels}\,\, and Yvik Swan\footnote{Unit\'e de Recherche en Math\'ematiques, 6, rue Richard Coudenhove-Kalergi, L-1359 Luxembourg}}

\date{}
\maketitle
\vspace{-1cm}

\begin{center}
Universit\'e libre de Bruxelles and Universit\'e du Luxembourg
\end{center}

\begin{abstract}

 Stein operators are differential operators which arise
    within the so-called Stein's method for stochastic
    approximation. We propose a new mechanism for constructing
    such operators for arbitrary (continuous or discrete) parametric
    distributions with continuous dependence on the parameter. We
    provide explicit general expressions for location, scale and
    skewness families. We also provide a general expression for discrete
    distributions.  For specific choices of target distributions (including the
    Gaussian, Gamma and Poisson) we compare the
    operators hereby obtained with those provided by the classical
    approaches from the literature on Stein's method. We use
    properties of our operators to provide upper and lower variance bounds
    (only lower bounds in the discrete case) on functionals $h(X)$ of
    random variables $X$ following parametric distributions. These bounds
    are expressed in terms of the first two moments of the derivatives
    (or differences) of $h$.
    We provide general variance bounds for location, scale and skewness   families and apply our bounds to specific examples (namely the Gaussian,
    exponential, Gamma and Poisson distributions). The results obtained via
    our techniques are systematically competitive with, and sometimes
    improve on, the best bounds available in the literature.  
  
\end{abstract}

\noindent\emph{AMS 2000 subject classifications}: Primary 62E17; Secondary 60E15

\noindent\emph{Keywords}: Chernoff inequality; Cram\'er-Rao inequality; parameter of interest; Stein characterization; Stein's method.

 
\section{Introduction}
\label{sec:introduction}

Let $g$ be a given target density (continuous or discrete) and let $X \sim g$. Choose $d$  a probability metric (Kolmogorov, Wasserstein, total
variation, ...) and suppose that we aim to estimate the distance $d(W,
X)$ between the law of some random variable $W$ and that of $X$.
Stein's method (introduced in the pathbreaking \cite{S72, St86})
advocates to first construct a suitable differential operator $f
\mapsto \mathcal{T}(f,g)$ such that $X\sim g\Longleftrightarrow {\rm
  E}[\mathcal{T}(f,g)(X)]=0 \mbox{ for all } f\in\mathcal{F}(g),$ with
$\mathcal{F}(g)$ a specific ($g$-dependent) class of \emph{test
  functions}, and then to use estimates on $\delta_g= \sup_{f \in
  \mathcal{F}(g)} | {\rm E} \left[ \mathcal{T}(f,g)(W) \right]|$
(which, of course, is 0 if $W \sim g$) in order to estimate $d(W,
X)$. Although mainly reserved to Gaussian approximation
\cite{BC05,NP11,ChGoSh11} and Poisson approximation \cite{BaHoJa92},
the method has also been proven in recent years to be very powerful
for other types of approximation problems
\cite{NoPe09,Lu94,Pi04,Do12,GoRe12,PeRoRo12,PeRo11,ChFuRo11}.

The key to the success of the method lies in the properties of the
operator $\mathcal{T}(\cdot,g)$ which, if the operator is well chosen, not
only allow to obtain good estimates on the quantity $\delta_g$ but
also guarantee that these in turn yield precise information on the
probability distance $d(W, X)$.  There are several well-documented
ways to construct a suitable Stein operator including the so-called
\emph{generator method} introduced by \cite{G91,MR1035659} and the
so-called \emph{density approach} introduced in
\cite{St86,StDiHoRe04}. For instance if $g=\phi$  the standard Gaussian
density, then a routine application of the density approach gives the
first-order operator $\mathcal{T}(f,\phi)(x) = -f'(x) + x f(x)$, while  the generator approach
brings  the infinitesimal generator of the
Ornstein-Uhlenbeck process on $\R$, that is,   the second-order
operator $\tilde{\mathcal{T}}(f,\phi)(x) = -f''(x) 
+ x f'(x)$. 
If $g$ is the
rate-$1$ 
exponential distribution then suitable modifications of the density
approach provide the operators $\mathcal{T}(f,g)(x) = -f'(x) + f(x) $
and $\bar{\mathcal{T}}(f,g)(x) = -xf'(x) + (x-1) f(x)$; 
both have been used for exponential approximation problems
\cite{ChFuRo11,PeRo11}. See also \cite{Re04,LS11c,Do12,GoRe12,LS12a}
for more examples and details. 

Now consider a random variable $X \sim g$.  Stein operators
  allow, in essence,  to write general integration by parts formulas of the form
  \begin{equation}
    \label{eq:gen1}
    \mathrm{E}\left[ f(X) h'(X) \right] = \mathrm{E} \left[
      \mathcal{T}(f, g)(X) h(X)  \right]
  \end{equation}
which hold for all sufficiently regular test functions $f$ and
$h$. Setting $f = 1$  in  \eqref{eq:gen1}   and applying the
Cauchy-Schwarz inequality to the right-hand side we
deduce that
\[\frac{(\mathrm{E} \left[ h'(X) \right])^2 }{{\rm E} \left[ (\mathcal{T}(1,
    g)(X))^2 \right]}\le  \mathrm{E} \left[ (h(X))^2 \right]\]
 for all appropriate test functions $h$. This is a generalization of the celebrated Cram\'er-Rao inequality, with ${\rm E} \left[( \mathcal{T}(1,
    g)(X))^2 \right]$ being some form of Fisher information for $X$. 
In particular if $g= \phi$ is the density of a standard Gaussian random variable and $h$ has mean zero under $\phi$,
then $\mathcal{T}(1, \phi)(x)= -x$ and  this last result particularizes to
$({\rm E}[h'(X)])^2\leq{\rm Var}[h(X)]$.  Chernoff \cite{C80,C81} used a method involving
Hermite polynomials to prove that a converse inequality holds
in the case of a Gaussian, namely 
\begin{equation}\label{normalbounds}
({\rm E}[h'(X)])^2\leq{\rm Var}[h(X)]\leq{\rm E}[(h'(X))^2]
\end{equation}
with equality on both sides if and only if $h$ is
linear. Although Chernoff's technique of proof is tailored for a
Gaussian target,  an alternative proof hinging on the properties of
the Gaussian Stein operator was obtained in
\cite{Ch80} by Chen. Chen's approach was rapidly seen to be robust to a change in the
distribution, and  similar inequalities as \eqref{normalbounds} were
shown to hold for many other
distributions than the standard normal \cite{C82, K85}. The connection
between the so-called Stein type identities and extensions of the
bound \eqref{normalbounds} to arbitrary target distributions has  now
been explored in quite some detail
\cite{papadatos2001unified,C82,C81,Ch80,CPU94,CP95,borovkov1984inequality}. 

\begin{remark}
  In the case of  functionals of a Gaussian, the upper and lower bounds
  in \eqref{normalbounds} were shown in \cite{HK95} -- again through
  an argument specifically tailored for a Gaussian target -- to be the
  first terms in an infinite series expansion. Recently \cite{APP11}
  extended the latter series expansion for the lower bound from a
  Gaussian target to the class of Pearson distributions (and Ord
  distributions in the discrete case) by studying the properties of the
  Stein operators for distributions satisfying Pearson's (or Ord's)
  definition.
\end{remark}

\begin{remark}
 Bounds such as \eqref{normalbounds} and
its variations for alternative targets have deep theoretical and
practical implications. These are connected to the classical
isoperimetric problem and thus also to logarithmic Sobolev
inequalities \cite{logsob_book}.  Tight estimates on the constants in
these inequalities (called Poincar\'e constants) provide crucial
quantitative information on the properties of the distribution of the
random variable.
\end{remark}

In this paper we develop (Section \ref{sec:mainresult}) a new
mechanism -- which we call the \emph{parametric approach} -- for
building Stein operators in terms of the \emph{parameters of interest}
(location parameter, scale parameter, skewness parameter, ...)  of the
target distribution $g$.  More precisely, given a target $g$ and a
parameter $\theta$ we identify a maximal class $\mathcal{F}(g;
\theta)$ of test functions {(maximal because the conditions are
  minimal)} and a differential operator $\mathcal{T}_{
  \theta}(\cdot,g)$ such that 
\[X \sim g(\cdot; \theta)
\Longleftrightarrow {\rm E} \left[ \mathcal{T}_{\theta}(f,g)(X)
\right]=0\]
 for all $f \in \mathcal{F}(g; \theta)$.  We show (Sections
 \ref{locmodel}-~\ref{sec:discr-param-distr}) that the
operators $\mathcal{T}_{\theta}(f, g)$ indeed generalize the classical
Stein operators from the literature. We then use these operators to propose
(Section \ref{sec:furtherwork}) an extension of \eqref{normalbounds}
to a wide variety of target distributions $g$. One of the strong
points of our  bounds   is that, when applied to specific
distributions such as the exponential, Gamma or Gaussian, the results
obtained via our technique are systematically competitive with (and
sometimes improve on) the best bounds available in the literature.
Detailed specific examples are provided and discussed throughout, and lengthy
proofs are deferred to the end of the paper (Section~\ref{sec:proofs}). 

\section{Parametric Stein operators} \label{sec:mainresult} 

Throughout  we let $\Theta\subseteq \R$ be a non-empty subset of
$\R$ and say that a  measurable function
$g:\R\times \Theta \rightarrow\R^+$ 
 forms a family of $\theta$-\emph{parametric densities} on $\R$ (with
 respect to some general $\sigma$-finite dominating measure $\mu$) if
 \begin{equation}
   \label{eq:19}
   \int g(x;\theta)d\mu(x)=1
 \end{equation}
 for all $\theta \in \Theta$. If in \eqref{eq:19} $\mu$ is the
 counting measure on the integers then we further have $0 \le
 g(x;\theta) \le 1$ for all $x$ and $\theta$.  For $\theta_0\in\Theta$
 ($\theta_0$ has of course the same parametric nature as~$\theta$), we
 denote by $\mathcal{G}(\R, \theta_0)$ the collection of
 $\theta$-parametric densities on $\R$ for which there exist a bounded
 neighborhood $\Theta_0\subset \Theta$ of $\theta_0$ and a $\mu$-integrable
 function $h:\R\rightarrow\R^+$ such that $g(x;\theta)\leq h(x)$ over
 $\R$ for all $\theta\in\Theta_0$.  Given $\theta_0\in \Theta$ and $g \in \mathcal{G}(\R,
 \theta_0)$,   we write $X \sim g(\cdot;
 \theta_0)$ to denote a random variable distributed according to the
 (absolutely continuous or discrete) probability law $x \mapsto g(x; \theta_0)$.

\begin{definition}\label{def}
  Let $\theta_0$ be an interior point of $\Theta$ and let $g \in
  \mathcal{G}(\R, \theta_0)$.  Define $S_{\theta} := \left\{ x \in \R \, |
  \, g(x; \theta)>0\right\}$ as the support of
  $g(\cdot; \theta)$. We
  define the class $\mathcal{F}(g;\theta_0)$ as the collection of
  functions $f:\R\times\Theta\rightarrow\R$ such that there
  exists $\Theta_0$  some neighborhood of $\theta_0$  where  the
  following three conditions are satisfied : 
\begin{enumerate}[(i)]
\item  there exists a constant $c_f \in \R$ (not depending on $\theta$) such that
  $\int  f(x;\theta)g(x;\theta)d\mu(x)=c_f$ for
  all $\theta\in\Theta_0$;  
\item for all $x\in S_{\theta}$ the mapping $\theta \mapsto f(\cdot ;
  \theta)g(\cdot; \theta)$ is differentiable in the sense of
  distributions over $\Theta_0$;
\item there exists an integrable function $h:
  \R \to \R^+$ such that  for all
  $\theta\in\Theta_0$ we have 
  $|\partial_{\theta}(f(x;\theta)g(x;\theta))|\leq h(x)$ over
  $\R$.
\end{enumerate}
 We define the {\em Stein
    operator} $\mathcal{T}_{\theta_0}:=\mathcal{T}_{\theta_0}(\cdot,
  g):\mathcal{F}(g;\theta_0) \rightarrow\R^*$ as
  \[\mathcal{T}_{\theta_0}(f, g)(x)=
    \frac{\partial_\theta(f(x;\theta)g(x;\theta))|_{\theta=\theta_0}}{g(x;\theta_0)},
\]
with the convention that $1/g(x;\theta_0) = 0$  outside  the support
$S_{\theta_0} \subseteq \R$  of $g(\cdot; \theta_0)$. 
\end{definition}

The conditions imposed in Definition \ref{def} are, in a sense, too
stringent, and minimal conditions on the test functions $f$ are
obtained simply by requiring that one can take derivatives (with
respect to $\theta$) under the integral sign. Conditions (ii) and
(iii) are natural sufficient assumptions for this manipulation to be
allowed; see \mbox{e.g.}  \cite{LC98}. 

We now state the main result of this section. The proof 
is provided in Section~\ref{sec:proofs}.  
\begin{theorem}[Parametric Stein characterization]\label{theo}  Fix an
  interior point $\theta_0 \in \Theta$. Let  $g \in
  \mathcal{G}(\R, \theta_0)$   
and  $Z_{\theta}$ be distributed according
  to $g(\cdot; \theta)$,
  and 
  let $X$ be a random variable taking values on $\R$.   Then the following two
  assertions hold.
\begin{itemize}
\item[(1)] If $X\stackrel{\mathcal{D}}{=} Z_{\theta_0}$, then ${
    \rm E}[\mathcal{T}_{\theta_0}(f,g)(X)]=0$ for all
  $f\in\mathcal{F}(g;\theta_0)$.
\item[(2)] If  the support $S_\theta:=S$ of $g(\cdot;\theta)$ does not depend on $\theta$ and if ${\rm E}[\mathcal{T}_{\theta_0}(f,g)(X)]=0$ for all 
  $f\in\mathcal{F}(g;\theta_0)$,
  then $X\,|\,X\in
    S\stackrel{\mathcal{D}}{=} Z_{\theta_0}.$
\end{itemize}
\end{theorem}

In the next sections we consider three well-known types of parameters,
namely location, scale and skewness (in each case for absolutely
continuous target distributions), and use Theorem \ref{theo} to
construct a selection of relevant and tractable Stein operators which
are, in a sense, natural with respect to the choice of parameter. We
will also see how to apply Theorem~\ref{theo} in the case of general
discrete distributions with continuous dependence on the parameter.

\subsection{Stein operators for location models}\label{locmodel}

Let the dominating measure $\mu$ be the Lebesgue measure on $\R$ (and write $dx$ for $d\mu(x)$). Let $ \Theta=\R$,  fix $\mu_0 \in \R$ (typically one takes
  $\mu_0=0$) and consider densities of the form 
  \begin{equation}\label{eq:8}
    g(x; \mu)=g_0(x-\mu), \mu\in\R,
  \end{equation}
for $g_0$ some positive function integrating to 1 over its support. We denote by 
  $\mathcal{G}_{\text{loc}}$ the collection of $g_0$'s for which 
  $\mu$-parametric densities of the form~\eqref{eq:8} 
  belong to $\mathcal{G}(\R, \mu_0)$. 

Clearly, in the present
context,  Condition~(i) of Definition \ref{def}  
holds most naturally for test functions of the form $f(x; \mu) =
f_0(x-\mu)$ 
for which we also have 
  \begin{equation}
    \label{eq:9}
    \partial_x (f_0(x- \mu) g_0(x-\mu)) = - \partial_{\mu} (f_0(x- \mu)  g_0(x-\mu))
  \end{equation}
  for all $(x, \mu) \in \R\times \R$  (we write $\partial_x$ and
  $ \partial_{\mu}$ the weak derivatives with respect to $x$ and
  $\mu$, respectively). 
    Let $g_0 \in \mathcal{G}_{\mathrm{loc}}$. We then  define  
  $\mathcal{F}_{\mathrm{loc}}(g_0; \mu_0)$  the collection of all
  $f_0: \R \to \R$ such that 
 (i) $\int_\R  f_0(x-\mu)
    g_0(x-\mu)  dx=\int_\R   f_0(x)
    g_0(x)  dx=c_{f_0} $ some finite constant;
 (ii) the mapping $x \mapsto
  f_0(x)g_0(x)$ is differentiable in the sense of distributions;
 (iii)  there exists an integrable function $h$  such
  that $\left|\left.\partial_{y} (f_0(y - \mu) g_0(y-\mu))
    \right|_{y=x}\right|\le h(x)$ over $\R$ for all $\mu\in
  \Theta_0$, some bounded neighborhood of $\mu_0$. 

  \begin{corollary}[Location-based Stein operator]\label{cor:oper-locat-models}
The conclusions of Theorem \ref{theo} apply to any location model of
the form \eqref{eq:8} with $g_0\in \mathcal{G}_{\mathrm{loc}}$ and
operator 
\begin{equation} \label{eq:location}
\mathcal{T}_{\mu_0;\mathrm{loc}}(f_0,g_0) : \R \to \R : x \mapsto  
\frac{\left.- \partial_y (f_0(y- \mu_0) g_0(y-\mu_0) ) \right|_{y=x}}{g_0(x-\mu_0)},
\end{equation}
for $f_0 \in \mathcal{F}_{\mathrm{loc}}(g_0; \mu_0)$ and with $\partial_y$ the
derivative in the sense of distributions with respect 
to $y$.   
 \end{corollary}



 Take $g_0(x) = \phi(x)$ the density of a
 $\mathcal{N}(0,1)$ random variable (which clearly belongs to
 $\mathcal{G}_{\mathrm{loc}}$). Then, for $\mu_0=0$
 and any weakly differentiable function $f_0\in\mathcal{F}_{\rm
   loc}(\phi;0)$, Corollary~\ref{cor:oper-locat-models} yields the
 operator
\[
    \mathcal{T}_{\rm loc}(f_0,\phi)(x)=-f_0'(x)+xf_0(x),
\]
which shows that the usual Stein operator associated with the normal
  distribution is, statistically speaking, associated with the location parameter. More generally, for $n \in
 \N_0$, define recursively the sequence of polynomials $H_0(x) = 1$,
 $H_{n+1}(x) = -H_n'(x) + xH_n(x)$ (that is, $H_n(x)$ is the $n$th
 Hermite polynomial) and consider functions of the form
 $f:\R\times\R\rightarrow\R:(x,\mu)\mapsto f(x; \mu):=
 H_n(x-\mu)f_0(x-\mu)$, where  $f_0:\R\rightarrow\R$ is chosen such
 that $f\in \mathcal{F}_{\mathrm{loc}}(\phi;0)$.  Restricting the operator
 $\mathcal{T}_{\rm loc}(\phi,\cdot)$ to this collection of $f$'s, we find   
 \begin{equation}
 \label{eq:25} 
 \mathcal{T}_{\mathrm{loc}}(f_0,\phi)(x)= -H_n(x)f_0'(x)+H_{n+1}(x)f_0(x), \quad n
     \ge 0.
   \end{equation}
 This family of  operators was discovered by
     \cite{GR05}.  

   Next, take $g_0(x) = e^{-x} \mathbb{I}_{[0,
         \infty)}(x)$ the rate-1 exponential density (which, as for
       the Gaussian, clearly belongs to
       $\mathcal{G}_{\mathrm{loc}}$). Again setting $\mu_0=0$ we get
     the operator
     \begin{equation}
       \label{eq:locstexp}
       \mathcal{T}_{\mathrm{loc}}(f_0,{\rm Exp}) = (-f_0'(x) + f_0(x))
\mathbb{I}_{[0, \infty)}(x) -f_0(0)\delta_{x=0}, 
     \end{equation}
with $\delta_{x=0}$ the Dirac delta at $x=0$ (recall that the
derivative in \eqref{eq:location} is the derivative in the sense of
distributions).   
This was first obtained in \cite{StDiHoRe04} and used in
\cite{ChFuRo11} under the restriction $f_0(0)=0$. More generally, when $g$ belongs to the
(continuous) \emph{exponential family} (see \cite{LC98} for a precise definition), these location-based manipulations allow to retrieve the known
operators discussed e.g. in  \cite{H78}, \cite{H82} or
\cite{LC98}. 


\subsection{Stein operators for scale models}\label{scamodel}

Let the dominating measure $\mu$ be the Lebesgue measure on $\R$ (and write $dx$ for $d\mu(x)$).  Let  $\Theta = \R_0^+$,     fix
$\sigma_0 \in \Theta$ (typically one takes $\sigma_0=1$) and consider densities of the form 
  \begin{equation}\label{eq:sca}
    g(x; \sigma)=\sigma g_0(\sigma x), \sigma\in\R_0^+,
  \end{equation}
for $g_0$ some positive function integrating to 1 over its support. We denote by 
  $\mathcal{G}_{\text{sca}}$ the collection of $g_0$'s for which 
  $\sigma$-parametric densities of the form~\eqref{eq:sca} 
  belong to $\mathcal{G}(\R, \sigma_0)$. 

Condition~(i) of Definition \ref{def}  here
holds most naturally for test functions of the form $f(x; \sigma) =
f_0(\sigma x)$ 
for which we have the relationship
  \begin{equation}
    \label{eq:sca2}
    \partial_x (x f_0(\sigma x) g_0(\sigma x)) =  \partial_{\sigma} (f_0(\sigma x) \sigma g_0(\sigma x))
  \end{equation}
  for all $(x, \sigma) \in \R\times \R_0^+$. 
    Let $g_0 \in \mathcal{G}_{\mathrm{sca}}$. We then  define  
  $\mathcal{F}_{\mathrm{sca}}(g_0; \sigma_0)$  the collection of all
  $f_0: \R \to \R$ such that 
 (i) $\int_\R  f_0(\sigma x) \sigma
    g_0(\sigma x)  dx=\int_\R   f_0(x)
    g_0(x)  dx=c_{f_0} $ some finite constant;
 (ii) the mapping $x \mapsto
  f_0(x)g_0(x)$ is differentiable in the sense of distributions;
 (iii)  there exists an integrable function $h$  such
  that $\left|\left.\partial_{y} (y f_0(\sigma y) g_0(\sigma y))
    \right|_{y=x}\right|\le h(x)$ over $\R$ for all $\sigma\in
  \Theta_0$, some bounded neighborhood of $\sigma_0$. 

  \begin{corollary}[Scale-based Stein operator]\label{cor:oper-scale-models}
   The conclusions of Theorem \ref{theo} apply to any scale model of
the form \eqref{eq:sca} with $g_0\in \mathcal{G}_{\mathrm{sca}}$ and
operator 
\[
\mathcal{T}_{\sigma_0;\mathrm{sca}}(f_0,g_0) : \R \to \R : x \mapsto  
\frac{\left. \partial_y (y f_0(\sigma_0y) g_0(\sigma_0y) ) \right|_{y=x}}{\sigma_0g_0(\sigma_0x)},
\]
  for $f_0 \in \mathcal{F}_{\mathrm{sca}}(g_0; \sigma_0)$  and  $\partial_y$ the
  derivative in the sense of distributions with respect 
to $y$.  
  \end{corollary}

 Take  $g_0(x) = \phi(x)$ the density of a $\mathcal{N}(0,1)$ (which
 clearly also belongs to $\mathcal{G}_{\mathrm{sca}}$), that is, this
time we consider the normal with the scale parameter as parameter of
interest. For $\sigma_0=1$ and any weakly differentiable function
$f_0\in\mathcal{F}_{\rm sca}(\phi;1)$, Corollary
\ref{cor:oper-scale-models} yields the operator 
 \begin{equation*}
    \mathcal{T}_{\rm sca}(f_0,\phi)(x)=x f_0'(x)-(x^2-1)f_0(x),
 \end{equation*}
which is (up to the minus sign) a particular case of \eqref{eq:25} for
$n=1$. 

Next take  $g_0(x) = e^{-x}\mathbb{I}_{[0, \infty)}(x)$
  (which also belongs to $\mathcal{G}_{\mathrm{sca}}$). Note in particular how the support $\R^+$ is invariant
under scale change. Applying Corollary \ref{cor:oper-scale-models} we get the
operator
\begin{equation*}
     \mathcal{T}_{\rm sca}(f_0,{\rm Exp})(x) =  (xf_0'(x)-(x-1)f_0(x))\mathbb{I}_{[0,\infty)}(x)
\end{equation*}
after setting   $\sigma_0=1$. This scale-based operator has first been exploited in
\cite{ChFuRo11}. More generally, choosing $g$ the probability density function (pdf) of a Gamma
distribution with
shape $a>0$   we obtain 
\begin{equation*}
     \mathcal{T}_{\rm sca}(f_0,{\rm Gamma})(x) =  (xf_0'(x)-(x - 
     a)f_0(x))\mathbb{I}_{[0,\infty)}(x), 
\end{equation*}
a variant of the Gamma operator used, e.g., by \cite{NoPe09}.


\subsection{Stein operators for skewness models}

Let the dominating measure $\mu$ be the Lebesgue measure on $\R$ (and write $dx$ for $d\mu(x)$).
Contrarily to location and scale models which are defined in a
canonical way, there exist several distinct skewness models and no
canonical form of asymmetry. A popular family are the
sinh-arcsinh-skew (SAS) laws of \cite{JP09}. These laws are a
particular case of the construction given in \cite{LP10} who consider
monotone increasing diffeomorphisms $H_\delta:\R\rightarrow\R$ indexed
by the skewness parameter $\delta\in\R$ in such a way that $H_0(x)=x$
is the only odd transformation. Letting $g_0$ be a symmetric positive
function integrating to 1 over its support, this ensures that the
resulting densities
\begin{equation}\label{skew}
  g(x;\delta)=(H_\delta)'(x)g_0(H_\delta(x)),
\end{equation}
with $(H_\delta)'(x)=\partial_x H_\delta(x)$, are indeed skewed if
$\delta$ differs from 0, value for which  the initial symmetric density
$g_0$ is retrieved. The sinh-arcsinh transformation corresponds to
$H_\delta(x)=\sinh(\sinh^{-1}(x)+\delta)$. We shall call LP-densities
the skewed distributions~(\ref{skew}).

For these skew distributions, let $\Theta = \R$, and fix $\delta_0 \in
\Theta$. LP-skewness models possess densities of the
form~(\ref{skew}), and for a given transformation $H_\delta$ we denote
by $\mathcal{G}_{\text{skew}}(H_\delta)$ the collection of $g_0$'s for
which $\delta$-parametric densities of the form~\eqref{skew} belong to
$\mathcal{G}(\R, \delta_0)$. In order to produce the desired
operators, we however further need to add the condition that both
$\delta\mapsto H_\delta(\cdot)$ and $\delta\mapsto(H_\delta)'(\cdot)$
are differentiable in the sense of distributions.

Condition~(i) of Definition \ref{def}  here
holds naturally for test functions of the form $f(x; \delta)
=f_0(H_\delta(x))$. 
    Let $g_0 \in \mathcal{G}_{\mathrm{skew}}(H_\delta)$. We then  define  
  $\mathcal{F}_{\mathrm{skew}}(g_0;H_{\delta_0})$  the collection of all
  $f_0: \R \to \R$ such that 
 (i) $\int_\R  f_0(H_\delta(x)) (H_\delta)'(x)
    g_0(H_\delta( x))  dx=\int_\R   f_0(x)
    g_0(x)  dx=c_{f_0} $ some finite constant;
 (ii) the mapping $x \mapsto   f_0(x)g_0(x)$ is differentiable in the sense of distributions;
 (iii)  there exists an integrable function $h$  such
  that $\left|\partial_{\delta} ( f_0(H_\delta(x))(H_\delta)'(x) g_0(H_\delta(x)))\right|\le h(x)$ over $\R$ for all $\delta\in
  \Theta_0$, some bounded neighborhood of $\delta_0$. 

  \begin{corollary}[LP-skewness-based Stein operator]\label{cor:oper-skew-models}
 The conclusions of Theorem \ref{theo} apply to any LP-skewness model of
the form \eqref{skew} with $g_0\in \mathcal{G}_{\mathrm{skew}}(H_\delta)$ and
operator 
\[
\mathcal{T}_{H_{\delta_0};\mathrm{skew}}(f_0,g_0) : \R \to \R : x \mapsto  
\frac{\left. \partial_\delta (f_0(H_\delta(x)) (H_\delta)'(x)
    g_0(H_\delta( x)) ) \right|_{\delta=\delta_0}}{(H_{\delta_0})'(x)g_0(H_{\delta_0}(x))}
\]
for $f_0 \in \mathcal{F}_{\mathrm{skew}}(g_0;H_{\delta_0})$.
\end{corollary}

Given a continuous density $g_0$ we define (as in  \cite{JP09})
 the SAS-skew-model  
\[  
  g(x; \delta) =(1+x^2)^{-1/2} C_\delta(x)g_0(S_\delta(x))
\]
 where  
$S_\delta(x)=\sinh(\sinh^{-1}(x)+\delta)$ and
$C_\delta(x)=\cosh(\sinh^{-1}(x)+\delta)$  ($g(x; \delta)$ clearly belongs to
$\mathcal{G}(\R,\delta_0)$ for any $\delta_0\in\R$).  Then we have the
relationship
 \begin{equation}
   \label{eq:2}
    \partial_x \left(
     C_{\delta}(x) f_0\left( S_{\delta}(x) \right) g_0\left( S_{\delta}(x) \right)\right)=\partial_{\delta} \left(f_0
     \left( S_{\delta}(x) \right) \frac{C_{\delta}(x)}{\sqrt{1+x^2}} g_0\left( S_{\delta}(x) \right)\right) 
 \end{equation}
for all  weakly differentiable functions
 $f_0\in\mathcal{F}_{\mathrm{skew}}(\phi;S_{\delta_0})$. 
Specifying  Corollary
 \ref{cor:oper-skew-models} to this skewing mechanism we get the operator 
\[   \mathcal{T}_{\mathrm{skew}}(f_0,g_0)(x)=C_{\delta_0}(x)f_0'(S_{\delta_0}(x))+\left(\frac{S_{\delta_0}(x)}{C_{\delta_0}(x)}   
  +C_{\delta_0}(x) \frac{g_0'(S_{\delta_0}(x))}{g_0(S_{\delta_0}(x))}\right)f_0(S_{\delta_0}(x)).
\]
Fixing $\delta_0 = 0$ the above becomes 
\[
  \mathcal{T}_{\mathrm{skew}}(f_0,g_0)(x)= \sqrt{1+x^2}f_0'(x)+\left(\frac{x}{\sqrt{1+x^2}}   
  + \sqrt{1+x^2} \frac{g_0'(x)}{g_0(x)}\right)f_0(x).
\]
Further specifying $g_0=\phi$ the standard Gaussian pdf and  taking $f_0(x) = \sqrt{1+x^2}f_1(x)$
with $f_1$ some suitable function we obtain  
\[
   \mathcal{T}_{\phi}(f_1)(x)= (1+x^2)f_1'(x) -(x^3 -x)f_1(x),
\]
which seems to be a new operator for the Gaussian distribution. 
 
\subsection{Discrete parametric distributions}
\label{sec:discr-param-distr}

Let the dominating measure $\mu$ be the  counting measure on $\Z$. Let $\Theta \subset \R$, and   fix
$\theta_0 \in \Theta$. Define 
$\mathcal{G}_{\text{dis}}$ as the
collection of $\theta$-parametric discrete densities $g \in
\mathcal{G}(\Z, \Theta)$  such that 
$g(\cdot;\theta):\Z\to[0,1]$ has support
 $S=[N]:=\{0, \ldots,
N\}$ for some $N\in \N_0\cup\{\infty\}$ not depending on
$\theta$ and that the function $\theta \mapsto g(x; \theta)$ is weakly
differentiable around $\theta_0$ at all $x \in [N]$.

 Condition (i) of Definition \ref{def}
here holds  for  test functions of the form 
\begin{equation}
  \label{eq:26} 
 f(x; \theta) =   \dfrac{D_x^{+} \left( f_0(x)\dfrac{g(x;
       \theta)}{g(0;\theta)}\right)}{g(x; \theta)}  
\end{equation}
(with $ D_x^{+} (f(x)) = f(x+1)-f(x)$ the
forward difference operator). 
The combination of $D_x^{+}$ and $\partial_{\theta}$ permits us to   exchange derivatives
with respect to the variable and 
  the parameter, that is, for $f$  of the form \eqref{eq:26}   we have the relation  
\[
    \partial_{\theta} (f(x;\theta) g(x; \theta)) =
    D_x^{+}(f_0(x)\partial_\theta (g(x;
      \theta)/g(0;\theta)))
\]
for all $(x, \theta) \in [N] \times \R$. 
 Let $g \in \mathcal{G}_{\mathrm{dis}}$. We then define
 $\mathcal{F}_{\mathrm{dis}}(g;\theta_0)$  the collection of all functions $f_0: 
\Z \to \R$ such that (i) $\sum_{x=0}^ND_x^{+}(f_0(x)\partial_\theta (g(x;
      \theta)/g(0;\theta)))<\infty$ and (ii) there exists a summable function
$h:\Z\to \R^+$ such that $\left|\Delta^+_x(f_0(x) \partial_u(g(x;
  u)/g(0;u))|_{u=\theta}) \right|\le h(x)$ over $\Z$ for all
$\theta \in \Theta_0$ some neighborhood of $\theta_0$. Note that here Condition (ii) of Definition \ref{def} is always satisfied since we use the forward difference. Moreover, for finite $N$, the above-mentioned sum is also finite, and we have $\sum_{x=0}^ND_x^{+}(f_0(x)\partial_\theta (g(x;
      \theta)/g(0;\theta)))=-f_0(0)$ which does not depend on $\theta$.

\begin{corollary}[Discrete Stein operator]\label{cor:discrete} 
 The conclusions of Theorem \ref{theo} apply to any discrete
 distribution  $g \in \mathcal{G}_{\mathrm{dis}}$ with 
operator 
\[
   \mathcal{T}_{\theta_0;\mathrm{dis}}(f_0,g_0) : \Z \to \R : x \mapsto   \dfrac{
    D_x^+\left(f_0(x)\left.\partial_\theta\big({g(x;
          \theta)}/{g(0;
          \theta)}\big)\right|_{\theta=\theta_0}\right)}{g(x;
    \theta_0)}
\]
for  $f \in \mathcal{F}_{\mathrm{dis}}(g;\theta_0)$.
\end{corollary}

   Take $g(x;\lambda)=e^{-\lambda}{\lambda ^x}/{x!}\, \mathbb{I}_{\N}(x)$, the
density of a Poisson $\mathcal{P}(\lambda)$ distribution. Clearly, $g$
belongs to $\mathcal{G}_{\mathrm{dis}}$ for all $\lambda \in \R_0^+$ and its support
$S=\N$ is independent of $\lambda$. Then, for $x \in \N_0$ we
have  
$  \left.\partial_\lambda\big({g(x;
          \lambda)}/{g(0;
          \lambda)}\big)\right|_{\lambda=\lambda_0} =
      {\lambda_0^{x-1}}/{(x-1)!} 
$
so that 
\[
   \mathcal{T}_{\mathrm{dis}}(f_0,\mathcal{P}(\lambda_0))(x) 
=e^{\lambda_0}\left(f_0(x+1)-\frac{x}{\lambda_0}f_0(x)\right)\mathbb{I}_{\N}(x),
\]
which is (up to the scaling factor) the usual operator for the
Poisson.  Setting $g(x; p)=(1-p)^xp\, \mathbb{I}_{\N}(x)$ the geometric $Geom(p)$ distribution, we get
\begin{equation*}
  \mathcal{T}_{\mathrm{dis}}(f_0,Geom(p))(x)  =  \frac{1}{p}\left((x+1)
  f_0(x+1)-\frac{x}{1-p}f_0(x)\right)\mathbb{I}_{\N}(x).
\end{equation*}
Finally, for the binomial $Bin(n, p)$, we obtain the  $p$-characterizing operator 
\begin{equation*}
  \mathcal{T}_{p;\mathrm{dis}}(f_0,Bin(n,p))(x) =(1-p)^{-n-2}\left((n-x)
  f_0(x+1)-\frac{1-p}{p}xf_0(x)\right)\mathbb{I}_{[n]}(x).
\end{equation*}
These last two operators are not new, and can be obtained (up to
scaling factors) via the generator approach \cite{Ho04}.

\section{Variance bounds} \label{sec:furtherwork} 

Consider a   $\theta$-parametric density $g \in
\mathcal{G}(\R, \theta_0)$ with associated Stein class $\mathcal{F}(g;
\theta_0)$ and operator $\mathcal{T}_{\theta_0}(\cdot, g)$ at some
point $\theta_0 \in \Theta$. Suppose, for simplicity, that the support
$S_{\theta}$ of $g(\cdot; \theta)$ is a real interval with closure
$\bar{S_{\theta}} = [a,b]$ for $-\infty \le a < b \le \infty$, where
$a=a_{\theta}$ and $b=b_{\theta}$. (If $\mu$ is the counting measure
then $S = \left\{ a, a+1, \ldots, b-1, b \right\}$.) 

We single out the subclass $\mathcal{F}_1(g; \theta_0)
\subset\mathcal{F}(g; \theta_0)$ (often written $\mathcal{F}_1$,
whenever no ambiguity ensues) of test functions such that, for all
$\theta$ in some bounded neighborhood  $\Theta_0$
 of $\theta_0$,  (i) $f(x; \theta) \ge
0$ over $\R$, (ii) $\int_{\R} f(x; \theta) g(x; \theta) d\mu(x) = 1$ 
and (iii) the function  
\begin{equation}\label{eq:exchfun} 
  \tilde{f}(x;\theta)=\frac{1}{g(x;\theta)}\int_a^x\partial_\theta(f(y;\theta)g(y;\theta))d\mu(y)
\end{equation}
satisfies the boundary conditions 
\begin{equation}
  \label{eq:33}
 \tilde{f}(a; \theta) g(a; \theta) =
\tilde{f}(b; \theta) g(b; \theta) = 0
\end{equation}
for all $\theta \in
\Theta_0$. 
For $f \in \mathcal{F}_1(g;
\theta_0)$ the function $g^{\star}(x; \theta) = f(x; \theta)g(x;
\theta)$ is again  a $\theta$-parametric density and we have the  ``exchange of
  derivatives'' relation
\begin{equation}
  \label{eq:7}
\partial_{\theta} ( f(x; \theta) g(x; \theta) ) =
\partial_x (  \tilde{f}(x; \theta) g(x;
    \theta) )  \mbox{ for all } x \in \R \mbox{ and all } \theta \in \Theta_0.
\end{equation}
For ease of reference we call the pair $(f, \tilde{f})$
exchanging around $\theta$. If $\mu$ is the counting measure then the
derivative $\partial_x$ in \eqref{eq:7} is to be replaced with the forward difference
operator $D_x^{+}$. 

\subsection{The continuous case}
\label{sec:continuous-case}

Take the dominating measure $\mu$  the Lebesgue measure (and write $dx$ for $d\mu(x)$). All
distributions considered in this section are absolutely continuous
with respect to $\mu$, and we use the superscript $'$ to indicate a (classical) strong derivative.


Our generalized variance bounds are provided in the following
theorem, whose proof (given in Section~\ref{sec:proofs}) strongly relies on the crucial condition~(\ref{eq:33}) and on the Stein characterizations of Theorem~\ref{theo}.

\begin{theorem}\label{theo:main}
Let $g \in \mathcal{G}(\R, \theta_0)$ and $X \sim g(\cdot;
\theta_0)$.  Choose $f \in \mathcal{F}_1(g; \theta_0)$ and  let $(f, \tilde f)$ be
exchanging around $\theta$. Let $X^{\star}_{f, \theta_0} \sim g^\star(\cdot;
\theta_0) = f(\cdot; \theta_0)
g(\cdot; \theta_0)$. We write 
$\varphi_{\theta_0,g^{\star}}(x) :=  \partial_{\theta} (\log \left(
  g^{\star}(x; \theta) \right)) \big|_{\theta=\theta_0} (=  {\mathcal{T}_{\theta_0}(f,g)(x)}/{f(x;\theta_0)})
$
the score function of $X^{\star}_{f, \theta_0}$ and 
$    \mathcal{I}(\theta_0, g^{\star}) := {\rm E}  [  (
    \varphi_{\theta_0,g^{\star}}(X^{\star}_{f,\theta_0})  )^2 ] 
$ its  Fisher information.
Then 
 \begin{equation} \label{eq:bound1}
{\rm Var}\left[h(X^{\star}_{f, \theta_0})\right]\ge \frac{\left({\rm E}\left[h'(X)\tilde{f}(X;
      \theta_0) \right]\right)^2}{  \mathcal{I}(\theta_0, g^\star)} 
\end{equation}
for all  $h \in C_0^{1}(\R)$.  If, furthermore,  $ x \mapsto 
\varphi_{\theta_0,g^\star}(x) $ is  strictly monotone and strongly differentiable over its support  then 
 \begin{equation} \label{eq:bound2}
{\rm Var}\left[h(X^{\star}_{f, \theta_0})\right] \le  {\rm E}\left[
  \frac{(h'(X))^{2}}{-\varphi'_{\theta_0, g^\star}(X)}\tilde{f}(X; \theta_0)\right] 
\end{equation}
for all  $h\in C_0^{1}(\R)$. Moreover equality holds in
\eqref{eq:bound1} and \eqref{eq:bound2}  if and
only if $h(x) \propto  \varphi_{\theta_0, g^\star}(x)$ for 
all $x$.
\end{theorem}
\begin{remark}
The upper bound in (\ref{eq:bound2}) is always
positive. Indeed, 
 first
observe that if $\varphi_{\theta_0,g^{\star}}$  is a diffeomorphism
then it is, in particular, strictly monotone over the support
$S_{\theta_0}$ and the function $x \mapsto \partial_{\theta}(f(x;
\theta) g(x; \theta))|_{\theta=\theta_0}$ changes sign exactly once
(because $\int_a^b \partial_{\theta}(f(x; \theta) g(x;
\theta))|_{\theta=\theta_0}dx = 0 $).  Hence if
$\varphi_{\theta_0,g^{\star}}$  is monotone increasing (resp.,
decreasing) then 
$\tilde{f}(x;\theta_0) \le 0$ (resp., $\tilde{f}(x;\theta_0) \ge 0$)
for all $x\in S_{\theta_0}$ so that the upper bound in
(\ref{eq:bound2}) is positive.
\end{remark}

A natural choice of test function in Theorem
  \ref{theo:main} is the constant function $f(x; \theta) = 1$, for
  which $g^{\star}(x; \theta)=g(x; \theta)$ and thus $X^{\star}_{f,
    \theta_0} \stackrel{\mathcal{L}}{=}X$. This
  choice is not always permitted : if the support of $g$ depends on
  the parameter and if the density does not cancel at the edges of the
  support then condition \eqref{eq:33} cannot be satisfied and our
  proofs break down. This is easily seen in the case of the rate-1 exponential
  distribution with location parameter $\mu$. There the dependence of the
  support on the parameter implies the appearance of  a Dirac delta
 in the expression of the location operator \eqref{eq:locstexp}; as a
 consequence Theorem \ref{theo:main} does not apply to this particular
 case. We contend that this breakdown  is not a drawback of our approach but rather
  one of its strengths and we will see that, despite this restriction,
  the bounds we obtain are as good as if not better than those
  already available in the literature (see the discussion at the end
  of this section). 

 In practice,
  the problem is avoided by assuming that the support of $g(\cdot;
  \theta)$ is either open or does not depend on $\theta$.   Then $f(x;
  \theta)=1$ is permitted and, using \eqref{eq:9}, \eqref{eq:sca2} and
\eqref{eq:2} (which are the specific versions of
  \eqref{eq:7} with respect to the different roles of the parameters
  considered in Section \ref{sec:mainresult}) we  obtain explicit
  forms for the exchanging functions $\tilde{f}$, and thus explicit
forms of the variance bounds from Theorem \ref{theo:main}. 
\begin{proposition}[Location, scale and skewness variance bounds]  \label{propos:apli}
Consider a $\theta$-parametric   density $g \in \mathcal{G}(\R, \theta_0)$ and let $X\sim g(\cdot;\theta_0)$.
  \begin{enumerate}
  \item  {\rm (Location-based variance bounds)}  Let $\theta=\mu \in \R$ be a
    location parameter and $g(x; \mu) = g_0(x-\mu)$  a location
    model for $g_0\in C_0^1(S)$ with open support $S$. Then the  exchanging function for  $f(x;\mu) =
    f_0(x-\mu)\in\mathcal{F}_{\rm loc}(g_0;\mu_0)$ around $\mu$ is
    $\tilde{f}(x; \mu) = -f_0(x-\mu)$. The location-score function (expressed in terms of $y=x-\mu$) is
\[
      \varphi_{g_0, \mathrm{loc}}(y) = -\frac{g_0'(y)}{g_0(y)}\mathbb{I}_{S}(y).
\]
 If 
 $\varphi_{g_0, \mathrm{loc}}$ is strictly monotone and strongly differentiable on $S$, then the location-based variance bounds read
 \begin{equation}
   \label{eq:3}
    \frac{\left({\rm E}\left[h'(X) \right]\right)^2}{  \mathcal{I}_{\rm loc}(g_0)} \leq {\rm Var}\left[h(X)\right] \leq   {\rm E}\left[
  \frac{(h'(X))^{2}}{ \varphi_{g_0, \mathrm{loc}}'(X-\mu_0)}\right] 
 \end{equation}
 for $h\in C_0^1(\R)$, with  $\mathcal{I}_{\rm loc}(g_0):={\rm E}\left[\left(\varphi_{g_0, \mathrm{loc}}(X-\mu_0)\right)^2\right]$. 
 
\item  {\rm (Scale-based variance bounds)}  Let $\theta=\sigma \in \R^{+}_0$ be a
  scale parameter and $g(x; \sigma) = \sigma g_0(\sigma x)$  a scale
  model for $g_0\in C_0^1(S)$ with either open support $S$ or support $S$ invariant under scale change. Then the  exchanging function for  $f(x;\sigma) = f_0(\sigma
  x)\in\mathcal{F}_{\rm sca}(g_0;\sigma_0)$ around $\sigma$ is
  $\tilde{f}(x; \sigma) = \frac{x}{\sigma} f_0(\sigma x)$. The
  scale-score function (expressed in terms of $y=\sigma x$) is 
\[
    \varphi_{g_0, \mathrm{scale}}(y) = 
   \frac{1}{\sigma}\left(1+  y\frac{g_0'( y)}{g_0( y)}\right)\mathbb{I}_S(y).
\]
  If  $\varphi_{g_0, \mathrm{scale}}$ is strictly monotone and strongly differentiable on  $S$,
  then the scale-based variance bounds read
   \begin{equation}
     \label{eq:6}
      \frac{\left({\rm E}\left[ h'(X)X \right]\right)^2}{\sigma_0^2\mathcal{I}_{\rm sca}(g_0)} \leq {\rm Var}\left[h(X)\right] \leq   {\rm E}\left[
  \frac{(h'(X))^{2}X}{-\sigma_0^2\varphi_{g_0, \mathrm{scale}}'(\sigma_0 X)}\right]  
   \end{equation}
 for $h\in C_0^1(\R)$, with $\mathcal{I}_{\rm sca}(g_0):={\rm
   E}\left[\left(\varphi_{g_0, \mathrm{scale}}(\sigma_0X)\right)^2\right]$.  

\item  {\rm (SAS-based variance bounds)}  Let $\theta=\delta\in \R$ be a
  skewness parameter and $g(x; \delta) =  {C_{\delta}(x)}/{
    \sqrt{1+x^2}} g_0(S_{\delta}(x))$  the SAS-skewness model for $g_0\in C_0^1(S)$ with  open support $S$. Then
  the  exchanging function for  $f(x;\sigma) =
  f_0(S_{\delta}(x))\in\mathcal{F}_{\rm skew}(g_0;S_{\delta_0})$
  around $\delta$ is $\tilde{f}(x; \delta) =
  \sqrt{1+x^2}f_0(S_{\delta}(x))$. The skewness-score function (expressed in terms of $y=S_\delta(x)$) is 
\[
    \varphi_{g_0, \mathrm{ skew}}(y) =
     \left(\frac{y}{C_{\delta}(S_{\delta}^{-1}(y))}+C_{\delta}(S_{\delta}^{-1}(y))\frac{g_0'(y)}{g_0(y)}\right)\mathbb{I}_S(y).
\]
  If  $\varphi_{g_0,\mathrm{
      skew}}(x)$ is  monotone and strongly differentiable on $S$,
  then the SAS-based variance bounds read
\begin{equation}
  \label{eq:18}
  \frac{\left({\rm E}\left[ h'(X)\sqrt{1+X^2} \right]\right)^2}{
  \mathcal{I}_{\rm skew}(g_0)} \leq {\rm Var}\left[h(X)\right] \leq
{\rm E}\left[
  \frac{(h'(X))^{2}\sqrt{1+X^2}}{-C_{\delta_0}(X)\varphi_{g_0,\mathrm{
          skew}}'(S_{\delta_0}(X))}\right]
\end{equation}
      for $h\in C_0^1(\R)$, with $\mathcal{I}_{\rm skew}(g_0):={\rm
        E}\left[\left(\varphi_{g_0, \mathrm{ skew}}(S_{\delta_0}(X))\right)^2\right]$.
\end{enumerate} 
The lower bounds in \eqref{eq:3},  \eqref{eq:6} and \eqref{eq:18}
 hold without condition on the monotonicity of the score 
function. In all cases the bounds are tight, in the sense that
equality holds if and only if the test function $h$ is proportional to
the score function. 
\end{proposition}


In what follows, we shall apply Proposition~\ref{propos:apli} to three examples of probability laws, namely the Gaussian, the exponential and the Gamma. We consider all three examples as location-scale models, but we apply the SAS-skewing mechanism only to the Gaussian distribution (as the others are already skewed over $\R$).

Once again take  $g_0(x) = \phi(x) = (2\pi)^{-1/2}e^{-x^2/2}$ the standard Gaussian
density. Then, of course, $g_0'(x)/g_0(x)=-x$ and $f=1$ belongs to $\mathcal{F}_1$ for any type of parameter. Applying Proposition
\ref{propos:apli}  for $\mu_0=0$ (location case), $\sigma_0=\sigma$ (scale case) and $\delta_0=0$ (skewness case) we get 
\[
  \varphi_{\phi, \mathrm{loc}}(x) = x, \quad  \varphi_{\phi, \mathrm{sca}}(x) =
  \frac{1}{\sigma}(1-x^2) \mbox{ and } \varphi_{\phi, \mathrm{skew}}(x) = \frac{-x^3}{\sqrt{1+x^2}}.
\]
Only the  location score function is a ``sensible'' diffeomorphism (indeed, the derivative of the skewness score vanishes at the origin, leading to an infinite upper bound). Simple computations yield
\[
  \mathcal{I}_{\rm loc}(\phi) = 1, \quad  \mathcal{I}_{\rm sca}(\phi) = \frac{2}{\sigma^2}
  \mbox{ and }  \mathcal{I}_{\rm skew}(\phi) = 3 -\sqrt{\frac{e\pi}{2}}
 \mbox{Erfc}(1/\sqrt2) \approx 2.34432 =: \kappa.
\]
We thus sequentially obtain the location-based variance bounds
\begin{equation*}
    \left(   {\rm E} \left[ h'(X) \right]  \right)^2 \le{\rm Var} \left[ h(X)
    \right] \le {\rm E} \left[ \left( h'(X) \right)^2 \right],
\end{equation*}
with equality if and only if $h$ is linear (this is the well-known
bound~(\ref{normalbounds}); moreover, adding a scale parameter $\sigma$ in this location setting results in dividing both the upper and lower bound by $\sigma^2$)   as well as the scale-based bound
\[
     \frac{1}{2}({\rm E}[Xh'(X)])^{2}\leq{\rm Var}[h(X)] 
\]
with equality if and only if $h(x)\propto 1-x^2$ (this  bound is
 given in \cite{K85}) and also the skewness-based bound
\[
  \frac{\left(   {\rm E} \left[ \sqrt{1+X^2} h'(X) \right]  \right)^2 }{\kappa}\le{\rm Var} \left[ h(X)
    \right] 
\]
with equality if and only if $h(x) \propto x^3/\sqrt{1+x^2}$. 

Next take $g_0(x) =  e^{- x} \mathbb{I}_{[0,\infty)}(x)$ the rate-$1$ exponential
density; here $f=1$ is only permitted in the scale case and we have
$g_0'(x)/g_0(x)=-1$ (for $x > 0$). Thus, by 
Proposition \ref{propos:apli} for $\sigma_0=\lambda$  we get 
\[
  \varphi_{Exp, \mathrm{sca}}(x) =  \frac{1}{\lambda}(1-x)\mathbb{I}_{[0,\infty)}(x).\] 
This scale-score function is clearly a 
diffeomorphism. Also $\mathcal{I}_{\mathrm{sca}}(Exp) = \frac{1}{\lambda^2}$, which yields  the   
scale-based variance bounds 
\begin{equation}
  \label{eq:scVBLS}
\left(\mathrm{E} \left[X h'(X) \right]\right)^2\le
\mbox{Var} \left[     h(X) \right] \le  \frac{1}{\lambda} \mbox{E} \left[ X(h'(X))^2  \right].
\end{equation}
For the sake of comparison, \cite{C82} proposes the lower and upper bounds 
\begin{equation}\label{eq:cacvb} 
  \left(\mathrm{E} \left[ X h'(X) \right]\right)^2 \le \mbox{Var}\left[ h(X) \right]
  \le \frac{1}{\lambda^2} \mbox{Var} \left[ h'(X) \right]+
  \frac{1}{\lambda} \mathrm{E} \left[ X (h'(X))^2 \right];
\end{equation}
while  \cite{K85} proposes 
 \begin{equation}\label{eq:klavb} 
 \left(\mathrm{E} \left[ X h'(X) \right]\right)^2 \le \mbox{Var}\left[ h(X) \right]
  \le \frac{  4}{\lambda^2}\mathrm{E} \left[ ( h'(X))^{2} \right].
 \end{equation}
The lower bound in both these seminal papers concurs with ours from
\eqref{eq:scVBLS}. Our upper bound is evidently a strict improvement
on  \eqref{eq:cacvb}. It also improves on~\eqref{eq:klavb} in several cases. Indeed, a simple integration by parts in our upper bound (provided that $h\in C_0^2(\R)$) allows to rewrite it under the form 
\begin{equation*}
\frac{1}{\lambda^2}\left({\rm E}[(h'(X))^2]+2{\rm E}[Xh'(X)h''(X)]\right).
\end{equation*}
Whenever the second term is zero (e.g., for $h(x)=x$) or negative (e.g., for $h(x)=\sqrt{x}$), our bound is better than~\eqref{eq:klavb}.

Finally take $g_0(x) = \frac{1}{\Gamma(a)}x^{a-1}e^{-x}\mathbb{I}_{[0,\infty)}(x)$ the pdf of
a Gamma distribution with shape $a>0$. Here $f=1$
is permitted in both location and scale cases if $a>1$ and reserved to the scale
case for $a\leq1$. For the sake of clarity we will only consider the
case $a>1$.  We have $g_0'(x)/g_0(x) = \frac{(a-1-
  x)}{x}$. Applying Proposition \ref{propos:apli} under the respective
restrictions on $a$ and for $\mu_0=0$ (location case) and $\sigma_0=b$ (scale case),   
 we get 
\[
  \varphi_{Gamma, \mathrm{loc}}(x) =  \frac{-a+1+ x}{x}\mathbb{I}_{[0,\infty)}(x) \quad \mbox{and} \quad
  \varphi_{Gamma, \mathrm{sca}}(x) = \frac{1}{b}(a -x) \mathbb{I}_{[0,\infty)}(x).
\]
Both score functions are 
diffeomorphisms (on $\R_0^{+}$). Also
\[
  \mathcal{I}_{\mathrm{loc}}(Gamma) = \left\{\begin{array}
  {ll}
 \frac{1}{a-2}&\mbox{if } a>2\\
  \infty&\mbox{if } 1<a\leq2
  \end{array}\right. \quad\mbox{and} \quad \mathcal{I}_{\mathrm{sca}}(Gamma) = \frac{a}{b^2}.
 \] 
This yields  the following : location-based bounds
\begin{equation}
  \label{eq:locvbga}
(a-2) \left(\mathrm{E} \left[ h'(X) \right]\right)^2\le \mbox{Var} \left[
    h(X) \right]\le \frac{1}{a-1}\mathrm{E}\left[(h'(X))^2X^2\right]  
\end{equation}
and scale-based bounds 
\begin{equation}
  \label{eq:scavbga}
  \frac{1}{a}  \left(\mathrm{E} \left[X h'(X) \right]\right)^2  \le
\mbox{Var} \left[ h(X) \right] \le  \frac{1}{b}\mbox{E} \left[ X(h'(X))^2\right].
\end{equation}
On the one hand  \cite{C82}
only proposes a lower bound (which concurs with ours). On the other
hand, \cite{K85} proposes for $a>2$
\begin{equation}
  \label{eq:klvbga}
  \max \left( \frac{a-2}{b^2} \left(\mathrm{E} \left[h'(X) \right]\right)^2,
    \frac{1}{a}\left(\mathrm{E} \left[X h'(X) \right]\right)^2   \right)   \le
\mbox{Var} \left[ h(X) \right] \le  \frac{1}{b} \mbox{E} \left[ X(h'(X))^2\right].
\end{equation}
The upper  bound coincides with that in~\eqref{eq:scavbga}, while both candidates for the lower bounds are given in \eqref{eq:locvbga} and \eqref{eq:scavbga}, respectively (for a true comparison, we need to add a scale parameter in the lower location bound  \eqref{eq:locvbga}, resulting in a division by $b^2$). 

We conclude this section by  determining conditions on $g$ and $\theta$ for which 
the bound \eqref{eq:bound2} takes on the form
\begin{equation}
  \label{eq:17}
  \mbox{Var}(h(X)) \le d\,  \mbox{E} \left[ (h'(X))^2 \right]
\end{equation}
for some positive constant  $d$. If the special case $f = 1$ is admissible then, trivially,   
$ d= d_{g, \theta_0}  = \sup_{x \in S} (-{\tilde{f}(x;
    \theta_0)}/{ \varphi'_{\theta_0, g^{\star}}(x)
    } )
$
plays the required role, and the question becomes that of determining
conditions under which this constant is finite.  Specializing to the
case of a location model
we obtain the following intuitive sufficient condition. 
\begin{proposition}\label{proop:notat-defin-main-3}
Let $g$ be a continuous density with open support and let $X \sim g$.
If the function  $x \mapsto (\log g(x))'$ is strict monotone
decreasing and if there exists
$\epsilon>0$ such that $
-(\log g(x))'' \ge \epsilon>0$ then  \eqref{eq:17} holds with $d_{g,
  \mu_0} = \frac{1}{\epsilon}$.
\end{proposition}
\begin{proof}
Take a  location model $g(x; \mu) = g(x-\mu)$ with constant test
  function   $f(x; \mu) = 1$.  Then $\tilde{f}(x; \mu) = -1$  and
  we compute 
  \begin{equation*}
 \frac{\tilde{f}(x;
  \mu_0)}{-  \varphi'_{\mu_0, g^{\star}}(x)
 }  =   \frac{1}{-\frac{g''(x-\mu_0)}{g(x-\mu_0)} + \left(\frac{g'(x-\mu_0)}{g(x-\mu_0)}\right)^2}
  = \frac{1}{-(\log 
  g(x-\mu_0))''}.
  \end{equation*}
The conclusion follows from   \eqref{eq:17}.
\end{proof}
Note that the assumptions of Proposition
\ref{proop:notat-defin-main-3} hold if $g(x) = e^{-\psi(x)}$ for
$\psi(x)$ a strict convex function, i.e. if $g$ is strongly
unimodal on $\R$. We hereby recover Lemma 2.1 from \cite{K85}.  In particular if
$g(x) = (2\pi \sigma^2)^{-1/2}e^{-x^2/(2\sigma^2)}$ is the $\mathcal{N}(0,\sigma^2)$  then
$\epsilon=1/\sigma^2$ and we re-obtain the well-known upper bound
$  \mbox{Var} \left( h(X) \right) \le \sigma^2 \mathrm{E} \left[
    (h'(X))^2 \right]. $

\subsection{The discrete case}
\label{sec:discrete-case}

Take as  dominating measure $\mu$ the counting measure. For $f$ and $g$ two functions such
that $\sum_{x=a}^b D_x^+(f(x)g(x))<\infty$ and $ f(b+1)g(b+1) = f(a)g(a) = 0 $, we have the discrete integration
by parts formula
\[
  \sum_{x=a}^b \left( D_x^{+} (f(x)) \right)g(x+1)   = -  \sum_{x=a}^b
  f(x) \left( D_x^{+} (g(x)) \right).
\]
The boundary condition  \eqref{eq:33} therefore allows deduce the
following partial discrete counterpart to Theorem~\ref{theo:main}, whose proof is left to the reader.

\begin{theorem}
  Let $g \in \mathcal{G}(\Z, \theta_0)$  and $X\sim g(\cdot;\theta_0)$. 
Choose $f \in \mathcal{F}_1(g; \theta_0)$ and  let $(f, \tilde f)$ be
exchanging around $\theta$. Let $X^{\star}_{f, \theta_0} \sim g^\star(\cdot;
\theta_0) = f(\cdot; \theta_0)
g(\cdot; \theta_0)$. We write 
$\varphi_{\theta_0,g^{\star}}(x) :=  \partial_{\theta} (\log \left(
  g^{\star}(x; \theta) \right)) \big|_{\theta=\theta_0} (=
{\mathcal{T}_{\theta_0}(f,g)(x)}/{f(x;\theta_0)}) 
$
the score function of $X^{\star}_{f, \theta_0}$ and 
$    \mathcal{I}(\theta_0, g^{\star}) := {\rm E}  [  (
    \varphi_{\theta_0,g^{\star}}(X^{\star}_{f,\theta_0})  )^2 ] 
$ its  Fisher information.
Then 
 \begin{equation}     \label{eq:34}
{\rm Var}\left[h(X^{\star}_{f, \theta_0})\right]\ge \frac{\left({\rm E}\left[D_x^{+}(h(x))|_{x=X}\tilde{f}(X;
      \theta_0) \right]\right)^2}{  \mathcal{I}(\theta_0, g^\star)} 
\end{equation}
for all  $h$ with equality if and only if $h(x)\propto \varphi_{\theta_0,g^{\star}}(x)$. 
\end{theorem}
Take  $g(x; \lambda) = e^{-\lambda}
\lambda^x/x!\mathbb{I}_\N(x)$ the pdf of the Poisson distribution. Then we  have 
$  \partial_{\lambda} g(x; \lambda) =- D_x^{+} \left(\frac{x}{\lambda} g(x; \lambda)\right)$;
in particular $1 \in \mathcal{F}_1$ because  $\tilde{1}(x;\lambda)g(x;\lambda)= \frac{x}{\lambda}g(x;\lambda)$ indeed cancels at the
edges of the support of $g$. Also we compute   $\varphi_{\lambda, g}(x) =
(-1+\frac{x}{\lambda})\mathbb{I}_\N(x)$ and $\mathcal{I}(\lambda, g) = 1/\lambda
$. Applying \eqref{eq:34} we conclude 
\[
  \mbox{Var}\left[ h(X) \right] \ge  \frac{1}{\lambda} {\rm E} \left[ XD_x^{+}(h(x))|_{x=X} \right]^2,
\]
with equality if and only if $h(x) \propto -1+x/\lambda$ on $\N$.
This last bound is, to the best of our knowledge, new.
\section{Proofs}
\label{sec:proofs}

\begin{proof}[Proof of Theorem \ref{theo}] 
  (1) Since Condition~(iii) allows for differentiating \mbox{w.r.t.}
  $\theta$ under the integral in Condition~(i) and since
  differentiating w.r.t. $\theta$ is allowed thanks to Condition~(ii),
  the claim follows immediately.  

\noindent (2)   We prove the claim in the continuous case (and write $dx$ for $d\mu(x)$). The
discrete case follows exactly along the same lines. 
  Define, for
  $A\subseteq\R$, the mapping
\[
f_A:\R\times\Theta_0\rightarrow\R :(x,\theta)\mapsto
\frac{1}{g(x;\theta)}\int_{\theta_0}^\theta l_A(x;u)g(x;u)du
\]
with  $l_A(x;u):=\left(\mathbb{I}_A(x)-{\rm P}(Z_u\in A\right))\mathbb{I}_{S}(x),$
where $  {\rm P}(Z_u\in B)=\int_\R\mathbb{I}_B(x)g(x;u)dx$
for $B \subseteq\R$.  Note that ${\rm P}(Z_u\in S)=1$ for all
$u\in\Theta_0$, since the support does not depend on the parameter of
interest. 
We claim that $f_A$ belongs to $\mathcal{F}(g;\theta_0)$. 
If this holds true  the conclusion follows since then,
by hypothesis,
\begin{equation*}
  {\rm E}[\mathcal{T}_{\theta_0}(f_A, g)(X)]= {\rm E}[l_A(X;\theta_0)]={\rm E}[\mathbb{I}_{A\cap
  S}(X)-{\rm P}(Z_{\theta_0}\in
A)\mathbb{I}_{S}(X)]=0
\end{equation*}
and thus 
\begin{equation*}
  \mathrm{P}(X \in A \, | \, X \in S) = {\rm P}(Z_{\theta_0}\in
A)
\end{equation*}
for all measurable $A\subset \R$.

To prove the claim first note that
\begin{equation*}
  \int_\R f_A(x;\theta)g(x;\theta)dx
  =\int_{\theta_0}^\theta\int_{S} l_A(x;u)g(x;u)dxdu
\end{equation*}
by Fubini's theorem, which can be
applied for all $\theta\in \Theta_0$ since in this case there exists
a constant $M$ such that
\begin{equation*}
  \int_{\R}\mathbb{I}_{(\theta_0,\theta)}(u)\int_{S}|l_A(x;u)|g(x;u)dxdu\leq|\theta-\theta_0|\le
M
\end{equation*}
for all $\theta \in \Theta_0$. We also have, by definition of $l_A$, that
\begin{align*}
   \int_{S} l_A(x;u)g(x;u)dx &   = {\rm P}(Z_u\in A\cap S)-{\rm
     P}\left(Z_u\in A\right){\rm P}(Z_u \in S) \\
   & 				 = 0.
\end{align*}
Hence $f_A$ satisfies Condition~(i). Condition~(ii) is easily
checked. Regarding Condition~(iii),  one sees  that 
$ \left.\partial_t\left(f_A(x; t)g(x;
      t)\right)\right|_{t=\theta} = l_A(x; \theta)g(x;
  \theta).$
By boundedness of the function $l_A(\cdot; \theta)$ and by definition of the class
$\mathcal{G}(\R, \theta_0)$ we know that  $
|l_A(x; \theta)g(x; \theta)|$ can be bounded by an integrable function $h(x)$ uniformly in $\theta\in\Theta_0$. Hence $f_A$ satisfies
Condition~(iii). We have thus proved that
$f_A\in\mathcal{F}(g; \theta_0)$, and the conclusion follows.
\end{proof}

\begin{proof}[Proof of Theorem \ref{theo:main}]
For the sake of
readability, throughout the proof we  simply write $X^{\star}:=X^{\star}_{f, \theta_0}$ and
  $\varphi(x):=\varphi_{\theta_0, g^\star}(x)$.
  
We first prove the lower bound \eqref{eq:bound1}. Take  $f \in
\mathcal{F}_1(g; \theta_0)$. Using \eqref{eq:7} and the different assumptions (which are tailored
for the following to hold) we get, on the one hand  
\begin{align*}        {\rm E} \left[h(X)
  \mathcal{T}_{\theta_0}(f, g)(X)\right]   &  =   \int_a^b
  h(x) {\partial_\theta(f(x; \theta) g(x; \theta))|_{\theta_0}}dx    =
  \int_a^bh(x) {\partial_x (\tilde f(x; \theta_0) g(x; \theta_0)}) dx
  \nonumber \\
  &  = -  \int_a^b h'(x)  \tilde f(x; \theta_0) g(x; \theta_0) dx  
= - {\rm E} \left[h'(X) \tilde{f}(X; \theta_0)\right]
       \end{align*}
and, on the other hand,  (recall that  $ \mathcal{T}_{\theta_0}(f,
g)(x) = \varphi(x)f(x; \theta_0)$) 
\begin{align}      \left|{\rm E} \left[h(X)
    \mathcal{T}_{\theta_0}(f, g)(X)\right] \right| & =\left|{\rm E} \left[(h(X)-{\rm E}[h(X^\star)])
    \mathcal{T}_{\theta_0}(f, g)(X)\right] \right|\label{Steinhelp}\\
    & \leq {\rm E} \left[|h(X)-{\rm E}[h(X^\star)]|\,  |\varphi(X)|f(X; \theta_0)\right] \nonumber\\
  & \le \sqrt{{\rm E} \left[(h(X)-{\rm E}[h(X^{\star})])^2 f(X; \theta_0)\right]   {\rm E} \left[f(X;
      \theta_0) (\varphi(X))^2\right]}\label{CShelp}\\
     &= \sqrt{{\rm Var} [h(X^{\star})] \, \mathcal{I}(\theta_0,g^\star)},\nonumber \end{align}
where~(\ref{Steinhelp}) follows from the Stein characterization of
Theorem~\ref{theo} 
and~(\ref{CShelp})  from the Cauchy-Schwarz inequality (recall that
$f$ is positive).

We now prove the upper bound \eqref{eq:bound2} in the case where
$\varphi$ is strict monotone decreasing, the increasing case being proved
exactly in the same way.  Let 
  $\varphi^{-1}(x)$ denote the inverse  function of $\varphi$. Then direct manipulations involving the Cauchy-Schwarz
  inequality yield 
\begin{align*}
   {\rm Var} \left[h(X^{\star})\right]& =
    {\rm Var} \left[  \int_0^{\varphi(X^\star)}(h\circ
        {\varphi}^{-1})'(u)du \right] \le {\rm E} \left[ \left( \int_0^{\varphi(X^\star)}(h\circ
        {\varphi}^{-1})'(u)du\right)^2 \right]\label{caco}\\
&  \le  {\rm E} \left[\int_0^{\varphi(X^{\star})}1^2du
  \int_0^{\varphi(X^{\star})}\left((h\circ
    \varphi^{-1})'(u)\right)^2du\right]\nonumber\\
&  = {\rm E} \left[ \varphi(X^{\star})
    \int_0^{\varphi(X^{\star})}\left(\frac{h'(\varphi^{-1}(u))}{\varphi'
        (\varphi^{-1}(u))}\right)^2du\right].    \nonumber
\end{align*}
Note how the latter expression is always positive: negative values of $\varphi(X^\star)$ are multiplied by a negative integral (since a positive function is integrated over $(0,\varphi(X^\star))$). Now  let $x_0$ be the unique point  in  $(a,b)$ such that 
$\varphi(x_0)=0$ and let
$\varphi(a)=P^+$ and $\varphi(b)=-P^{-}$ for some $P^{\pm}\in\R\cup \left\{ \pm\infty \right\}$. Then,
pursuing the above, 
\begin{align*}
 {\rm Var} \left[h(X^{\star})\right]
& \le \int_a^{x_0}
\int_0^{\varphi(x)} \partial_{\theta}(f(x; \theta)
g(x;\theta))|_{\theta=\theta_0}\left(\frac{h'(\varphi^{-1}(u))}{\varphi'(\varphi^{-1}(u))}\right)^2 du dx \\
& \quad + \int_{x_0}^{b}
\int_0^{\varphi(x)} \partial_{\theta}(f(x; \theta)g(x;\theta))|_{\theta=\theta_0}\left(\frac{h'(\varphi^{-1}(u))}{\varphi'(\varphi^{-1}(u))}\right)^2
du dx. \end{align*}
Using Fubini (which is possible since all quantities involved are positive), we deduce 
\begin{align*}
{\rm Var} \left[h(X^{\star})\right] & \le \int_0^{P^{+}}\left(\frac{h'(\varphi^{-1}(u))}{\varphi' (\varphi^{-1}(u))}\right)^2
\left(\int_a^{\varphi^{-1}(u)} \partial_{\theta}(f(x; \theta)g(x;\theta))|_{\theta=\theta_0} dx\right)du \\
& \quad - \int_{-P^{-}}^{0}\left(\frac{h'(\varphi^{-1}(u))}{\varphi' (\varphi^{-1}(u))}\right)^2
\left(\int_{\varphi^{-1}(u)}^b \partial_{\theta}(f(x; \theta)g(x;\theta))|_{\theta=\theta_0} dx\right)du
\end{align*}
From  \eqref{eq:7}  we then get 
\begin{align*}
 {\rm Var} \left[h(X^{\star})\right]
& \le
\int_0^{P^{+}}\left(\frac{h'(\varphi^{-1}(u))}{\varphi'
    (\varphi^{-1}(u))}\right)^2 
\left(\int_a^{\varphi^{-1}(u)} \partial_{x}(\tilde f(x; \theta_0)g(x;\theta_0)) dx\right)du \\
& \quad -
\int_{-P^{-}}^{0}\left(\frac{h'(\varphi^{-1}(u))}{\varphi'
    (\varphi^{-1}(u))}\right)^2 
\left(\int_{\varphi^{-1}(u)}^b \partial_{x}(\tilde f(x; \theta_0)g(x;\theta_0)) dx\right)du\\
& =
\int_0^{P^{+}}\left(\frac{h'(\varphi^{-1}(u))}{\varphi'
    (\varphi^{-1}(u))}\right)^2 
  \widetilde{f}(\varphi^{-1}(u); \theta_0)g(\varphi^{-1}(u);\theta_0)du \\
& \quad +
\int_{-P^{-}}^{0}\left(\frac{h'(\varphi^{-1}(u))}{\varphi'
    (\varphi^{-1}(u))}\right)^2 
\widetilde{f}(\varphi^{-1}(u);
\theta_0)g(\varphi^{-1}(u);\theta_0)du. 
\end{align*}
Setting $y = \varphi^{-1}(u)$ in the above and changing
variables accordingly we  obtain 
\begin{align*}
  {\rm Var} \left[h(X^{\star})\right]
& \le \int_b^{a} \frac{(h'(y))^{2}}{\varphi'(y)} 
  \tilde{f}(y; \theta_0)g(y;\theta_0)dy = {\rm E}\left[
    \frac{(h'(X))^{2}}{-\varphi'(X)}\tilde{f}(X; \theta_0) 
\right], 
\end{align*}
which is the claim. 
\end{proof}

\vspace{1cm}

\noindent ACKNOWLEDGEMENTS\vspace{0.2cm}

\noindent Christophe Ley, who is also a member of ECARES, thanks the Fonds National de la Recherche
Scientifique, Communaut\'e Fran\c caise de Belgique, for support via a
Mandat de Charg\'e de Recherche.



 \bibliographystyle{abbrv}

\bibliography{biblio_ys_stein}

\begin{thebibliography}{10}

\bibitem{APP11}
G.~Afendras, N.~Papadatos, and V.~Papathanasiou.
\newblock An extended {S}tein-type covariance identity for the {P}earson family
  with applications to lower variance bounds.
\newblock {\em Bernoulli}, 17(2):507--529, 2011.

\bibitem{logsob_book}
C.~An{\'e}, S.~Blach{\`e}re, D.~Chafa{\"{\i}}, P.~Foug{\`e}res, I.~Gentil,
  F.~Malrieu, C.~Roberto, and G.~Scheffer.
\newblock {\em Sur les in\'egalit\'es de {S}obolev logarithmiques}, volume~10
  of {\em Panoramas et Synth\`eses [Panoramas and Syntheses]}.
\newblock Soci\'et\'e Math\'ematique de France, Paris, 2000.
\newblock With a preface by Dominique Bakry and Michel Ledoux.

\bibitem{MR1035659}
A.~D. Barbour.
\newblock Stein's method for diffusion approximations.
\newblock {\em Probab. Theory Related Fields}, 84(3):297--322, 1990.

\bibitem{BC05}
A.~D. Barbour and L.~H.~Y. Chen.
\newblock {\em An introduction to Stein's method}, volume~4 of {\em Lect. Notes
  Ser. Inst. Math. Sci. Natl. Univ. Singap.}
\newblock Singapore University Press, Singapore, 2005.

\bibitem{BaHoJa92}
A.~D. Barbour, L.~Holst, and S.~Janson.
\newblock {\em Poisson approximation}, volume~2 of {\em Oxford Studies in
  Probability}.
\newblock The Clarendon Press Oxford University Press, New York, 1992.
\newblock Oxford Science Publications.

\bibitem{borovkov1984inequality}
A.~Borovkov and S.~Utev.
\newblock On an inequality and a related characterization of the normal
  distribution.
\newblock {\em Theory of Probability \&amp; Its Applications}, 28(2):219--228,
  1984.

\bibitem{C82}
T.~Cacoullos.
\newblock On upper and lower bounds for the variance of a function of a random
  variable.
\newblock {\em Ann. Probab.}, 10(3):799--809, 1982.

\bibitem{CP95}
T.~Cacoullos and V.~Papathanasiou.
\newblock A generalization of covariance identity and related
  characterizations.
\newblock {\em Math. Methods Statist.}, 4(1):106--113, 1995.

\bibitem{CPU94}
T.~Cacoullos, V.~Papathanasiou, and S.~A. Utev.
\newblock Variational inequalities with examples and an application to the
  central limit theorem.
\newblock {\em Ann. Probab.}, 22(3):1607--1618, 1994.

\bibitem{ChFuRo11}
S.~Chatterjee, J.~Fulman, and A.~Roellin.
\newblock Exponential approximation by exchangeable pairs and spectral graph
  theory.
\newblock {\em ALEA}, 8:1--27, 2011.

\bibitem{Ch80}
L.~H.~Y. Chen.
\newblock An inequality for multivariate normal distribution.
\newblock Technical report, MIT, 1980.

\bibitem{ChGoSh11}
L.~H.~Y. Chen, L.~Goldstein, and Q.-M. Shao.
\newblock {\em Normal approximation by {S}tein's method}.
\newblock Probability and its Applications (New York). Springer, Heidelberg,
  2011.

\bibitem{C80}
H.~Chernoff.
\newblock The identification of an element of a large population in the
  presence of noise.
\newblock {\em Ann. Statist.}, 8(6):1179--1197, 1980.

\bibitem{C81}
H.~Chernoff.
\newblock A note on an inequality involving the normal distribution.
\newblock {\em Ann. Probab.}, 9(3):533--535, 1981.

\bibitem{Do12}
C.~D\"obler.
\newblock Stein's method of exchangeable pairs for absolutely continuous,
  univariate distributions with applications to the polya urn model.
\newblock Preprint, arXiv:1207.0533, 2012.

\bibitem{GR05}
L.~Goldstein and G.~Reinert.
\newblock Distributional transformations, orthogonal polynomials, and {S}tein
  characterizations.
\newblock {\em J. Theoret. Probab.}, 18(1):237--260, 2005.

\bibitem{GoRe12}
L.~Goldstein and G.~Reinert.
\newblock Stein's method and the beta distribution.
\newblock Preprint, arxiv:1207.1460, 2012.

\bibitem{G91}
F.~G\"otze.
\newblock On the rate of convergence in the multivariate clt.
\newblock {\em Ann. Probab.}, 19(2):724--739, 1991.

\bibitem{Ho04}
S.~Holmes.
\newblock Stein's method for birth and death chains.
\newblock In {\em Stein's method: expository lectures and applications},
  volume~46 of {\em IMS Lecture Notes Monogr. Ser.}, pages 45--67. Inst. Math.
  Statist., Beachwood, OH, 2004.

\bibitem{HK95}
C.~Houdr\'e and A.~Kagan.
\newblock Variance inequalities for functions of {G}aussian variables.
\newblock {\em J. Theoret. Probab.}, 8(1):23--30, 1995.

\bibitem{H78}
H.~M. Hudson.
\newblock A natural identity for exponential families with applications in
  multiparameter estimation.
\newblock {\em Ann. Statist.}, 6:473--484, 1978.

\bibitem{H82}
J.~T. Hwang.
\newblock Improving upon standard estimators in discrete exponential families
  with applications to poisson and negative binomial cases.
\newblock {\em Ann. Statist.}, 10:857--867, 1982.

\bibitem{JP09}
M.~C. Jones and A.~Pewsey.
\newblock Sinh-arcsinh distributions.
\newblock {\em Biometrika}, 96:761--780, 2009.

\bibitem{K85}
C.~A.~J. Klaassen.
\newblock On an inequality of {C}hernoff.
\newblock {\em Ann. Probab.}, 13(3):966--974, 1985.

\bibitem{LC98}
E.~L. Lehmann and G.~Casella.
\newblock {\em Theory of point estimation}, volume~31.
\newblock Springer-Verlag, New York, 1998.

\bibitem{LP10}
C.~Ley and D.~Paindaveine.
\newblock Multivariate skewing mechanisms: a unified perspective based on the
  transformation approach.
\newblock {\em Statist. Probab. Lett.}, 80:1685--1694, 2010.

\bibitem{LS11c}
C.~Ley and Y.~Swan.
\newblock Stein's density approach for discrete distributions with applications
  to information inequalities.
\newblock Preprint, arXiv:1211.3668, 2012.

\bibitem{LS12a}
C.~Ley and Y.~Swan.
\newblock Stein's density approach and information inequalities.
\newblock {\em Electron. Comm. Probab.}, 18(7):1--14, 2013.

\bibitem{Lu94}
H.~M. Luk.
\newblock {\em Stein's method for the gamma distribution and related
  statistical applications}.
\newblock PhD thesis, University of Southern California, 1994.

\bibitem{NoPe09}
I.~Nourdin and G.~Peccati.
\newblock Stein's method on {W}iener chaos.
\newblock {\em Probab. Theory Related Fields}, 145(1-2):75--118, 2009.

\bibitem{NP11}
I.~Nourdin and G.~Peccati.
\newblock {\em Normal approximations with Malliavin calculus : from Stein's
  method to universality}.
\newblock Cambridge Tracts in Mathematics. Cambridge University Press, 2012.

\bibitem{papadatos2001unified}
N.~Papadatos and V.~Papathanasiou.
\newblock Unified variance bounds and a stein-type identity.
\newblock {\em Probability and Statistical Models with Applications}, pages
  87--100, 2001.

\bibitem{PeRo11}
E.~Pek\"oz and A.~Roellin.
\newblock New rates for exponential approximation and the theorems of r\'enyi
  and yaglom.
\newblock {\em Ann. Probab.}, 39(2):587 -- 608, 2011.

\bibitem{PeRoRo12}
E.~Pek\"oz, A.~Roellin, and N.~Ross.
\newblock Degree asymptotics with rates for preferential attachment random
  graphs.
\newblock {\em Ann. Appl. Probab.}, 2012, to appear.

\bibitem{Pi04}
A.~Picket.
\newblock {\em Rates of convergence of $\chi^2$ approximations via Stein's
  method}.
\newblock PhD thesis, Lincoln College, University of Oxford, 2004.

\bibitem{Re04}
G.~Reinert.
\newblock Three general approaches to stein's method.
\newblock In {\em An introduction to {S}tein's method}, volume~4. Lecture Notes
  Series, Institute for Mathematical Sciences, National University of
  Singapore, 2004.

\bibitem{S72}
C.~Stein.
\newblock A bound for the error in the normal approximation to the distribution
  of a sum of dependent random variables.
\newblock In {\em Proceedings of the {S}ixth {B}erkeley {S}ymposium on
  {M}athematical {S}tatistics and {P}robability ({U}niv. {C}alifornia,
  {B}erkeley, {C}alif., 1970/1971), {V}ol. {II}: {P}robability theory}, pages
  583--602, Berkeley, Calif., 1972. Univ. California Press.

\bibitem{St86}
C.~Stein.
\newblock {\em Approximate computation of expectations}.
\newblock Institute of Mathematical Statistics Lecture Notes---Monograph
  Series, 7. Institute of Mathematical Statistics, Hayward, CA, 1986.

\bibitem{StDiHoRe04}
C.~Stein, P.~Diaconis, S.~Holmes, and G.~Reinert.
\newblock Use of exchangeable pairs in the analysis of simulations.
\newblock In P.~Diaconis and S.~Holmes, editors, {\em Stein's method:
  expository lectures and applications}, volume~46 of {\em IMS Lecture Notes
  Monogr. Ser}, pages 1--26. Beachwood, Ohio, USA: Institute of Mathematical
  Statistics, 2004.

\end{thebibliography}

\end{document}